\documentclass[11pt]{amsart}
\usepackage{tabularx,booktabs,tikz}
\usepackage{caption}
\usepackage{amsmath}
\usepackage{amsfonts}

\usepackage{amscd}
\usepackage{amsthm}
\usepackage{amssymb} 
\usepackage{latexsym}
\usepackage{eufrak}
\usepackage{euscript}
\usepackage{epsfig}
\usepackage{graphics}
\usepackage{array}
\usepackage{enumerate}
\usepackage{dsfont}
\usepackage{color}
\usepackage{wasysym}
\usepackage{hyperref}
\usepackage{pdfsync}

\newcommand{\Hmm}[1]{\leavevmode{\marginpar{\tiny%
$\hbox to 0mm{\hspace*{-0.5mm}$\leftarrow$\hss}%
\vcenter{\vrule depth 0.1mm height 0.1mm width \the\marginparwidth}%
\hbox to
0mm{\hss$\rightarrow$\hspace*{-0.5mm}}$\\\relax\raggedright #1}}}

\newtheorem{theorem}{Theorem}[section]
\newtheorem{lemma}[theorem]{Lemma}
\newtheorem{corollary}[theorem]{Corollary}

\newtheorem{remark}[theorem]{Remark}

\newtheorem{proposition}[theorem]{Proposition}

\begin{document}

\title[Existence of ground state solutions to some Nonlinear Schr\"{o}dinger equations on lattice graphs]{Existence of  ground state solutions to some Nonlinear Schr\"{o}dinger equations on lattice graphs}

\author{Bobo Hua}
\address{Bobo Hua: School of Mathematical Sciences, LMNS, Fudan University, Shanghai 200433, China; ; Shanghai Center for Mathematical Sciences, Jiangwan Campus, Fudan University, No. 2005 Songhu Road, Shanghai 200438, China.}
\email{bobohua@fudan.edu.cn}

\author{Wendi Xu}
\address{Wendi Xu: School of Mathematical Sciences, Fudan University, Shanghai 200433, China}
\email{wdxu19@fudan.edu.cn}

\begin{abstract}
	In this paper, we study the  nonlinear Schr\"{o}dinger equation
	$	-\Delta u+V(x)u=f(x,u) $
on the lattice graph $\mathbb{Z}^{N}$. Using the Nehari method, we prove that when $f$ satisfies some growth conditions and the potential function $V$ is periodic or bounded, the above equation admits a ground state solution. Moreover, we extend our results from $\mathbb{Z}^{N}$ to\label{key} quasi-transitive graphs.

\end{abstract}
\par
\maketitle

\bigskip

\section{introduction}
The nonlinear Schr{\"o}dinger type equations of the form
\begin{equation}\label{1}
	-\Delta u+V(x)u=f(x,u), \quad u \in H^{1}(\Omega),
\end{equation}
where $\Omega$ is a bounded domain in $\mathbb{R}^{N}$ or the whole Euclidean space, have been extensively studied during the past several decades. See for examples \cite{Euclidean1,Euclidean2,Euclidean3,linkingthm,LYY,Cao,Struwe,Wzq1} and the references therein. Since the above equations have variational structures, one can use the critical point theory to study them. The corresponding variational functional is 
$$ \Phi(u)= \frac{1}{2}\int_{\mathbb{R}^{N}} |\nabla u|^{2}+V(x)u^{2} \, dx
-\int_{\mathbb{R}^{N}}F(x,u) \, dx,$$
where $F(x,u)=\int_{0}^{u} f(x,t) dt$. The solutions corresponding to the least positive critical value of $\Phi$ are called ground state solutions of (\ref{1}).

In \cite{Wzq}, Y.Q. Li, Z.Q. Wang and J. Zeng considered two cases of potential: one is periodic and the other is bounded, namely $\lim_{|x| \to \infty} V(x)=\sup_{\mathbb{R}^{N}}V<+\infty$. 
In both cases, without compact embedding for $H^{1}(\mathbb{R}^{N})$, they established, under a Nehari type condition (see \cite{Nehari1,Nehari2}) and a super-quadratic condition on the nonlinearity $f$, the existence of ground state solutions to (\ref{1}). In their note \cite{Andrzej}, S. Andrzej and T. Weth presented a unified approach to the Nehari method and proved results similar to Theorem 2.1 and Theorem 3.1 in \cite{Wzq}.

Recently, many researchers paid attention to various differential equations on graphs.
In \cite{GLW1,GLW2,GLW3}, A. Grigor’yan, Y. Lin and Y.Y. Yang applied variational methods on graphs to various types of nonlinear elliptic equations. In particular, in \cite{GLW3}, they proved the existence of a strictly positive solution to $-\Delta u + Vu =f(x,u)$ if  $V(x) \to \infty $ as ${\rm d}(x,x_{0}) \to \infty$ for some $x_{0}\in \mathbb{V}$, where d$(x,x_{0})$ is the combinatorial distance, i.e., the minimal number of edges contained in paths connecting $x$ and $x_{0}$.
Via the Nehari method, Zhang, N. and Zhao, L., see \cite{ZL}, proved that, if $a(x)\to \infty$ as d$(x,x_{0})\to \infty$, then, for any $\lambda>1$,  the nonlinear Schrödinger equation $- \Delta u +(\lambda a(x)+1)u = |u|^{p-1}u$   admits a ground state solution on locally finite graphs. Both \cite{GLW3} and \cite{ZL} assumed that the potential functions tend to infinity when ${\rm d}(x,x_{0}) \to \infty$, while we prove the results without such assumptions for special graphs.

In this paper, we consider the Schrödinger type equation 
$$-\Delta u+V(x)u=f(x,u)$$ on graphs. We first introduce the basic setting on graphs. Let $G=(\mathbb{V},\mathbb{E})$ be a simple, undirected, locally finite graph, where $\mathbb{V}$ is the set of vertices and $\mathbb{E}$ is the set of edges. Two vertices $x$ and $y$ are called neighbours, denoted by $x\sim y$, if there is an edge connecting them, i.e., $\{x,y\}\in \mathbb{E}$. A graph is called locally finite if each vertex has finitely many neighbours.

Denote the set of functions on $G$ by $C(V)$.
The Laplacian on $G=(\mathbb{V},\mathbb{E})$ is defined as, for any function $u\in C(V)$ and $x\in \mathbb{V}$,
\begin{displaymath}
		\Delta u(x) :=  \sum_{y \in \mathbb{V},y \sim x}\big(u(y)-u(x)\big).
\end{displaymath}
Let $C_{c}(\mathbb{V})$ be the set of all functions of finite support, and $H^{1}(\mathbb{V})$ be the completion of $C_{c}(\mathbb{V})$ under the norm
\begin{displaymath}
	\|u\|_{H^{1}}:= \Big( \frac{1}{2}\sum_{x\in \mathbb{V}}\sum_{y\sim x } \big(u(y)-u(x)\big)^{2}+\sum_{x\in \mathbb{V}} u^{2}(x) \Big)^{1/2}.
\end{displaymath}
We denote by $\mathbb{Z}^{N}$ the standard lattice graph with the set of vertices
$$ \{x=(x_{1},\cdots,x_{N}):x_{i}\in \mathbb{Z}, 1 \leq i \leq N\}$$
and the set of edges
$$\Big\{ \{ x,y\}: x,y\in\mathbb{Z}^{N},\sum_{i=1}^{N}|x_{i}-y_{i}|=1\Big\}.$$ 

For $T\in \mathbb{N}$, a function $g$ on $\mathbb{Z}^{N}$ is called T-periodic if $g(x+Te_{i})=g(x),\forall x\in\mathbb{Z}^{N},1\leq i \leq N$, where $e_{i}$ is the unit vector in the i-th coordinate.

We are concerned with the existence of ground state solutions of the following Schr\"{o}dinger equation:
\begin{equation}\label{maineq}
	\left \{\begin{array}{ll}
		-\Delta u+V(x)u=f(x,u), & x \in \mathbb{V} \\	
		u\in H^{1}(\mathbb{V})&\,
	\end{array}\right. 
\end{equation}
\begin{theorem}\label{thm1}
	Consider the Schr\"{o}dinger equation (\ref{maineq}) on $\mathbb{Z}^{N}$. Suppose that the nonlinearity $f:\mathbb{Z}^{N}\times\mathbb{R}\to \mathbb{R}$  is continuous with respect to $u$ and satisfies the growth condition 
	\begin{equation}\label{g.c.}
		|f(x,u)| \leq a(|u|+|u|^{q-1}) \quad \textrm{ for some }  a>0 \textrm{ and } q>2.
	\end{equation} 
Moreover, we assume the following conditons hold:\\
	(i)\, $V(x)$ and $f(x,u)$ are T-periodic in  $x$ for any $u\in \mathbb{R}$
	and $V(x)>0$ for all $x\in\mathbb{Z}^{N}$.\\
	(ii)\,$f(x,u) = o(u)$ uniformly in $x$ as $u \to 0$.\\ 
	(iii)\,$u\mapsto \frac{f(x,u)}{|u|}$ is strictly increasing on $(-\infty,0)$ and $(0,+\infty)$ respectively.\\   
	(iv)\,$ \frac{F(x,u)}{u^2} \to \infty $ uniformly in $x$ as $|u| \to \infty $ with $F(x,u)= \int_{0}^{u} f(x,s) \,d s $.\\
	Then the equation (\ref{maineq}) has a ground state solution.
\end{theorem}

We also consider the case that $V$ is not periodic but bounded in the sense that $\lim_{|x|\to \infty} V(x)=\sup_{\mathbb{Z}^{N}} V < +\infty$.
\begin{theorem} \label{thm2}
		Consider the Schr\"{o}dinger equation (\ref{maineq}) on $\mathbb{Z}^{N}$. Suppose that the nonlinearity $f=f(u)$ is a continuous function satisfying the growth condition (\ref{g.c.}) and the following conditions:\\
	(i)\,$f(u) = o(u)$  as $u \to 0$.\\
	(ii)\,$u\mapsto \frac{f(u)}{|u|}$ is strictly increasing on $(-\infty,0)$ and $(0,+\infty)$ respectively.\\
	(iii)\,$\frac{F(u)}{u^2} \to \infty$ as $|u| \to \infty$.\\
	Moreover, we assume that $V$ is a bounded potential well, i.e., $$0 <V_{0}\leq  \inf_{\mathbb{Z}^{N}} V(x) \leq \sup_{\mathbb{Z}^{N}} V(x) = V_{\infty}< \infty \quad \text{with }\quad V_{\infty}=\lim_{|x| \to \infty} V(x). $$    
	Then the equation (\ref{maineq}) has a ground state solution.
\end{theorem} 

\begin{remark}\label{rmk1}
	(1) The potential function $V(x) \to \infty $ as ${\rm d}(x,x_{0}) \to \infty$ is assumed in \cite{GLW3,ZL} to prove the existence of ground state solutions to $-\Delta u+V(x)u=f(x,u)$ on locally finite graphs. We don't need this assumption on lattice graphs.\\
    (2) The above two theorems are discrete analogs of the results in \cite{Andrzej}. However, the Sobolev embedding in our setting is different from that in the Euclidean space, which allows us to remove the constraint that $q<\frac{2N}{N-2}$(subcritical) in the continuous case.  
 \end{remark}
We follow the proofs in the continuous setting \cite{Andrzej} to prove the theorems. In particular, we set up the framework of Nehari methods on graphs and adopt it to prove the existence of ground state solutions to the nonlinear Schr\"{o}dinger equation (\ref{maineq}).

Moreover, we extend the results from lattice graphs to quasi-transitive graphs (see \cite{woess}). $G$ is called a quasi-transitive graph if there are finitely many orbits for the action of $Aut(G)$ on $G$ where $Aut(G)$ is the set of automorphisms of $G$.

\begin{theorem}\label{thm3}
	Let $G=(\mathbb{V},\mathbb{E})$ be a quasi-transitive graph, $\Gamma \leq Aut(G)$. Suppose that the action of  $\Gamma$ on $G$ has finitely many orbits.	Then Theorem \ref{thm2}  holds on $G$. By replacing T-periodicity of $f$ and $V$ with $\Gamma$-invariance, Theorem \ref{thm1} holds on $G$.	
\end{theorem}

This paper is organized as follows. In Section 2, we recall the setting of function spaces on graphs. Then we clarify the functional framework for Nehari manifold in Section 3. The proofs of Theorem \ref{thm1} and \ref{thm2} are given in Section 4 and Section 5 respectively. Since the proofs of theorems on quasi-transitive graphs are similar to that on lattice graphs, we give a brief proof at the end of Section 5. Finally, the detailed proofs of theorems in Section 3 are contained in the Appendix.

\section{Preliminaries}
 Let $G=(\mathbb{V},\mathbb{E})$ be a locally finite graph. In this paper, we always assume that
 \begin{equation}\label{bdd geo}
 	\#\{y\in \mathbb{V}:y\sim x\} \leq C,\quad \forall x \in \mathbb{V},
 \end{equation} 
where $C$ is a uniform constant. For any function $u,v\in C(\mathbb{V})$, define the associated gradient form as
\begin{displaymath}
	\Gamma (u,v)(x) := \sum_{y \sim x} \dfrac{1}{2}\big(u(y)-u(x)\big)\big(v(y)-v(x)\big).
\end{displaymath}
Write $\Gamma(u) = \Gamma(u,u)$ and 
\begin{displaymath}
	| \nabla u |(x) := \sqrt{\Gamma (u)(x)} =
	\Big( \sum_{y \sim x} \frac{1}{2} \big(u(y)-u(x)\big)^{2} \Big)^{1/2}.
\end{displaymath}
Let $\mu$ be the counting measure on $\mathbb{V}$, i.e., for any subset $A\subset \mathbb{V}$, $\mu(A):=\#\{x:x\in A\}$. For any function $f$ on $\mathbb{V}$, we write
\begin{displaymath}
	\int_{\mathbb{V}} f d \mu := \sum_{x\in \mathbb{V}} f(x) 
\end{displaymath}
whenever it makes sense.

The corresponding energy is defined as
\begin{displaymath}
	\mathcal{E} (f) := \int_{\mathbb{V}} 	| \nabla f |^{2} d \mu
	=\frac{1}{2}\sum_{x\in \mathbb{V}}\sum_{y \sim x}\big(f(y)-f(x)\big)^{2} .
\end{displaymath}
By definition, $H^{1}(\mathbb{V})$ is a Hilbert space with the inner product given by
\begin{displaymath}
\langle u,v\rangle:= \int_{\mathbb{V}} \Gamma(u,v) + uv \,d\mu.
\end{displaymath}

Let $\ell^{p}(\mathbb{V},\mu),p\in[1,\infty],$ be the Hilbert space of $\ell^{p}$ summable functions on $\mathbb{V}$ w.r.t. the measure $\mu$. We write $\|\cdot\|_{p}$ as the $\ell^{p}(\mathbb{V},\mu)$ norm, i.e., 
\begin{equation}\notag
\|u\|_{p}:=
\left\{ 	\begin{aligned}
		   & \sum_{x\in\mathbb{V}} |u(x)|^{p},\quad 1\leq p < \infty,\\
		    &\sup_{x\in\mathbb{V}} |u(x)|,\quad \quad  p=\infty.
	       \end{aligned}
\right.
\end{equation}
By definition, 
$$\|u\|_{H^{1}}^{2} = \int_{\mathbb{V}} \big(|\nabla u|^{2} + u^{2}\big) d\mu=\mathcal{E}(u)+\|u\|_{2}^{2},\\$$
and
\begin{equation}\nonumber
	\begin{aligned}
		\mathcal{E}(u) &=\frac{1}{2}\sum_{x\in \mathbb{V}}\sum_{y \sim x}\big(u(y)-u(x)\big)^{2} 
		\leq \sum_{x\in \mathbb{V}}\sum_{y \sim x}\big(u^{2}(x)+u^{2}(y)\big)\\
		&\leq C\|u\|_{2}^{2}.
	\end{aligned}
\end{equation}
Hence, there exists a constant $C^{'}>0$ such that
\begin{equation}\label{equiv}
	1/C^{'}\|u\|_{2}\leq \|u\|_{H^{1}} \leq C^{'}\|u\|_{2}.
\end{equation}
As usual, we write
$$\|u\|_{2} \simeq \|u\|_{H^{1}}$$
to indicate (\ref{equiv}).
For a bounded uniformly positive function $V: \mathbb{V} \to \mathbb{R}$, i.e. there exist $C_{1},C_{2}>0$, such that 
\begin{equation}\label{bdd ptt}
	C_{1} \leq V(x) \leq C_{2},\quad \forall x\in \mathbb{V},
\end{equation}
we can define an equivalent norm on $H^{1}(\mathbb{V})$ as
\begin{displaymath}
	\|u\|:= \Big(\int_{\mathbb{V}}|\nabla u|^{2} +Vu^{2} d\mu\Big)^{1/2}.
\end{displaymath} 
One can show that
$$
\|u\|  \simeq  \|u\|_{H^{1}}  \simeq  \|u\|_{2}.
$$
We say that $u\in C(\mathbb{V})$ is a weak solution of (\ref{maineq}) if 
\begin{displaymath}
	\int_{\mathbb{V}} \big( \Gamma(u,\phi)+Vu\phi \big) d \mu = \int_{\mathbb{V}} f(x,u) \phi\,d \mu , \quad \forall \, \phi \in H^{1}.
\end{displaymath}
Since $C_{c}(\mathbb{V})$ is dense in $H^{1}$, the integeral by parts gives 
\begin{displaymath}
	\int_{\mathbb{V}} \big( -\Delta u + Vu \big)\phi \, d \mu = \int_{\mathbb{V}} f(x,u) \phi \,d \mu , \quad \forall \, \phi \in C_{c}(\mathbb{V}).
\end{displaymath}
For any given $y \in \mathbb{V}$, one can take the test function $\phi = \delta_{y}(x)$ and obtain
$$	-\Delta u(y)+V(y)u(y)-f(y,u(y)) = 0.  $$
Now that $y\in \mathbb{V}$ is arbitrarily chosen, we have the following proposition.
\begin{proposition}
	If $u$ is a weak solution of (\ref{maineq}), then $u$ is a pointwise solution.  
\end{proposition}

\section{Nehari Manifolds}
In this section, we introduce the functional framework of Nehari manifolds on graphs following the continuous setting in e.g. \cite{Andrzej,Euclidean4,Euclidean5,P.L.Lions} and the references therein. The proofs of the results will be given in the Appendix.

Let $G=(\mathbb{V},\mathbb{E})$ be a locally finite graph satisfying (\ref{bdd geo}), $V:\mathbb{V}\to\mathbb{R}$ be the potential function satisfying (\ref{bdd ptt}) and $f(x,u)$ satisfy (i)-(iv) in Theorem \ref{thm1} without assuming the T-periodicity of $f$ and $V$. We write $H^{1}=H^{1}(\mathbb{V})$. Define the variational functional as 
\begin{displaymath}
	\begin{aligned}\notag
		\Phi (u) &:= \frac{1}{2} \int_{\mathbb{V}} \big( |\nabla{u}|^{2} + V(x)u^{2} \big) d \mu - \int_{\mathbb{V}} F(x,u) d \mu \\
		&= : \frac{1}{2} \|u\|^{2} - I(u)
	\end{aligned}	
\end{displaymath}
\begin{lemma}\label{lem1}
	Each solution $u\in H^{1}$ of (\ref{maineq}) corresponds to a critical point of $\Phi$, i.e., $\Phi^{'}(u) = 0$. 
\end{lemma}
Denote the unit sphere in $H^{1}$ by $S$. Set $\mathcal{N}= \{ u\in H^{1}\backslash\{0\} : \Phi'(u)u=0 \}$. We call it Nehari manifold which is crucial for searching nontrivial critical points of $\Phi$. In the following Propositions \ref{prop1}-\ref{prop3}, we show that, for each $w\in S$, the ray $w\mapsto sw\,(s>0)$ has a unique intersection with $\mathcal{N}$. Besides, $\mathcal{N}$ is uniformly bounded away from zero. This allows us to establish a homeomorphic mapping between $\mathcal{N}$ and $S$, where $S$ is a Banach submanifold of $H^{1}$.
\begin{proposition}\label{prop1}
	Let $I(u) = \int_{\mathbb{V}} F(x,u) \, d \mu$ and $F(x,u)= \int_{0}^{u} f(x,s) \,d s $. Then\\
	(i) $I^{'}(u) = o(\|u\|) $ as $u \to 0$,\\
	(ii) $s \mapsto \frac{I^{'}(su)u}{s}$ is strictly increasing on $(0, \infty)$ for any $u \neq 0$,  and\\
	(iii)\,$\frac{I(su)}{s^{2}} \to \infty $ uniformly on weakly compact subsets of $H^{1} \backslash \{0\}$ as $s \to \infty$.
\end{proposition}

\begin{proposition}\label{prop2}
	 For any $ w\in H^{1}\backslash \{0\}$, define a function $\alpha_{w}(s):= \Phi(sw), s\in(0,\infty)$. We have the following:\\
	(i)\,For any given $w$, there exists $s_{w}>0$  such that  $\alpha'_{w}(s)>0$  for  $0<s<s_{w}$  and  $\alpha'_{w}(s)<0$ for $s>s_{w}$.\\
	(ii)\,There exists $\delta >0$ such that $s_{w} \geq \delta$ for all $w\in S$.\\
	(iii)\,For each compact subset $\mathcal{K} \subset S,$ there exists a constant $C_{\mathcal{K}}$ such that $s_{w} \leq C_{\mathcal{K}}$ for all $w \in \mathcal{K}.$
\end{proposition}

Define 
\begin{displaymath}
	\begin{aligned}
		\widehat{m} : H^{1}\backslash\{0\} &\to \mathcal{N} \\
	w &\mapsto 	\widehat{m}(w)=s_{w}w  
	\end{aligned}
\end{displaymath}
and
$$  \quad m := \widehat{m}|_{S} :S \to \mathcal{N}.$$ 
Based on the geometric features of $\Phi$ in the radial direction, we have the following observation.
\begin{proposition}  \label{prop3}

	The mapping $\,\widehat{m}$ is continuous and $m$ is a homeomorphism between $S$ and $\mathcal{N},$ with the inverse of $m$ given by $m^{-1}(u) = \frac{u}{\|u\|}$.
\end{proposition}

Define the functional $\widehat{\Psi}:  H^1 \setminus\{0\} \to \mathbb{R} $ as $\widehat{\Psi}(w) := \Phi(\widehat{m}(w))$
and $\Psi:=\widehat{\Psi } |_{S} = \Phi \circ m $. One can prove that $w\in S$ is a critical point of $\Psi$ if and only if  $u=s_{w}w \in \mathcal{N}$ is a critical point of $\Phi$.

\begin{proposition}\label{prop4}
	We have
	$$\widehat{\Psi} =
	\Phi \circ \widehat{m} \in C^{1}\big( H^{1}\backslash\{0\},\mathbb{R} \big)$$
	and 
	$$  \widehat{\Psi } '(w)z= \ \frac{\| \widehat{m}(w)  \|}{\|w\|}\widehat{\Phi} '(\widehat{m}(w) )z, 
	\quad  \quad \forall w,z\in H^1 \backslash \{0\}. $$
\end{proposition}

\begin{corollary}
	When restricted on $S$, 
	$$\Psi = \widehat{\Psi} \big|_{S} \in C^{1}(S,\mathbb{R})$$and $$\Psi'(w)z = \|m(w)\| \Phi'\big(m(w)\big)z, \quad \forall z \in T_{w}S.$$
\end{corollary}
Recall that a sequence $\{u_{n}\} \subset H^{1}$ is called a \textit{Palais-Smale sequence} if $\{ \Phi(u_{n}) \}$ is bounded and $ \Phi'(u_{n}) \to 0$. If $\Phi(u_{n}) \to c $ and $\Phi^{'}(u_{n}) \to 0$, we call $\{u_{n}\}$ a $(PS)_{c}-sequence$.

\begin{proposition}\label{prop5}
	(i) If  $\{ w_{n} \}\subset S$ is a Palais-Smale sequence for $\Psi$, then $\{ m(w_{n})\}$ is a Palais-Smale sequence for $\Phi$. If $\{u_{n}\}$ is a bounded Palais-Smale sequence for $\Phi$, then $\{m^{-1}(u_{n})\}$ is a Palais-Smale sequence for $\Psi$.\\
	(ii) $w$ is a critical point of $\Psi$ if and only if $m(w)$ is a nontrivial critical point of $\Phi$. 
\end{proposition}

\section{The proofs of Theorem 1.1 }
In this section, we prove Theorem \ref{thm1}. Firstly, we are concerned with the properties of the lattice graph $\mathbb{Z}^{N}$. It's worth pointing out that we haven't used the periodicity of $f$ and $V$ yet. 
\begin{proof}[Proof of Theorem \ref*{thm1}]
Set $c:= \inf_S \Psi$ and take a minimizing sequence $ \{ w_{n} \}\subset S$, i.e., $\Psi(w_{n}) \to c$. By Ekeland's Variational Principle, see \cite{Struwe,EK.LY,EK.2}, we may assume that  $\Psi'(w_{n}) \to 0$, then $\{w_{n}\}$ is a $(PS)_{c}-sequence$ for $\Psi$. 

	Let $u_{n} = m(w_{n}) \in \mathcal{N},$ we claim that $\{\|u_{n}\|\} $ is uniformly bounded. By Proposition \ref{prop5}, we know that
	$\Phi'(u_{n})\to 0$
	and 
	$$ \Phi(u_{n})=\Psi(w_{n})\to c \equiv \inf_{S} \Psi=\inf_{\mathcal{N}} \Phi, \quad \text{where  }   u_{n}=m(w_{n}).$$
	Suppose that $\{\|u_{n}\|\}$ is unbounded, then, by passing to a subsequence, we may assume
	that $\|u_{n}\| \to \infty $ as $n\to \infty.$ Set $v_{n}:=\frac{u_{n}}{\|u_{n}\|} \in S$, then it has a weakly convergent subsequence 
	$v_{n} \rightharpoonup v$ in $H^{1}(\mathbb{Z}^{N})$, for $S$ is weakly compact. \\
	For any $\phi \in C_{c}(\mathbb{Z}^{N}),$
	$$\langle v_{n},\phi\rangle \, \to \,\langle v,\phi\rangle.$$By taking $\phi=\delta_{x}$, one can see that $\{v_{n}\}$ is pointwise convergent, i.e., $v_{n}(x) \to v(x),\, \forall x \in \mathbb{Z}^{N}.$ 
	
	We claim that $\{ v_{n} \}$ doesn't converge to 0 in $\ell^{q}$ $(q>2)$. Suppose that $v_{n} \to 0$ in $\ell^{q}(\mathbb{Z}^{N})$. For any $\epsilon>0$, there exists  $C_{\epsilon}>0$ such that, for all $s>0,$
	\begin{displaymath}
		\begin{aligned}
			\Big| \int_{\mathbb{Z}^{N}} F(x,sv_{n}) d \mu \Big| 
			& \leq  \int_{\mathbb{Z}^{N}} \big| \int_{0}^{sv_{n}}  f(x,t) dt \big| d\mu  \\
			& \leq  \int_{\mathbb{Z}^{N}} |sv_{n}|\Big(  \epsilon \big| sv_{n} \big| + C_{\epsilon} \big|sv_{n}\big|^{q-1} \Big) \,d\mu  \\
			&\leq \epsilon s^{2} \|v_{n}\|^{2} + C_{\epsilon} s^{q } \|v_{n}\|_{q}^{q}\\
			&\to \epsilon s^{2} \quad \text{ as } n\to \infty.
		\end{aligned}
	\end{displaymath}
	Letting $\epsilon \to 0^{+}$, we get
	$$\Big| \int_{\mathbb{Z}^{N}} F(x,sv_{n}) d \mu \Big|  \to 0 \quad \quad 
	\text{as } n\to \infty .$$
	Hence,
	$$\Phi(sv_{n}) = \frac{1}{2} s^{2} - \int_{\mathbb{Z}^{N}} F(x,sv_{n}) d \mu \, \to \,\frac{1}{2} s^{2} ,  \quad \forall s>0.$$
	By definition, $\Phi(sv_{n}) \leq \Phi(u_{n})$ for all $s>0$. Besides, $\Phi(u_{n}) \to c$ as $n \to \infty.$  Hence,
	$$c+1 \geq \lim_{n\to \infty} \Phi(sv_{n}) = \frac{1}{2} s^{2}, \quad \forall s>0.$$
	Choosing $s> \sqrt{2(c+1)}$, we get a contradition. Therefore, $v_{n}  \nrightarrow 0$ in $\ell^{q}$, i.e.,  $$\varlimsup_{n} \|v_{n}\|_{q} = c >0 $$ for some constant $c>0$.
	By interpolation inequality, we have
	$$ \|v_{n}\|_{q} \leq \|v_{n}\|_{2}^{\frac{2}{q}} \cdot\|v_{n}\|_{\infty}^{\frac{q-2}{q}}.	$$
	Taking the limit from both sides, one can see
	$$\varlimsup_{n} \|v_{n}\|_{q} 
	\leq \varlimsup_{n} \|v_{n}\|_{2}^{\frac{2}{q}}\|v_{n}\|_{\infty}^{1-\frac{2}{q}}.$$
	Since 
	$$\|v_{n}\|_{2} \leq \|v_{n}\| =1\quad \text{and}\quad
	\varlimsup_{n} \|v_{n}\|_{q} = c >0, $$
	we have 
	$$  \varlimsup_{n} \|v_{n}\|_{\infty}^{1-\frac{2}{q}} \geq	\varlimsup_{n}\|v_{n}\|_{2}^{\frac{2}{q}}\|v_{n}\|_{\infty}^{1-\frac{2}{q}}\geq c >0.$$
	Hence, there exists a subsequence $\{v_{n}\}$ and a sequence $\{ y_{n}\} \subset \mathbb{Z}^{N}$ such that $|v_{n}(y_{n})| \geq c$ for each $n$. By translations, we define $\widetilde{v}_{n}:=v_{n}(y+k_{n}T)$ with $k_{n}=(k^{1}_{n},\dots k^{N}_{n})$ to ensure that $\{y_{n}-k_{n}T\} \subset \Omega$ where $\Omega=[0,T)^{N}\bigcap \mathbb{Z}^{N}$ is a bounded domain in $\mathbb{Z}^{N}$. Then, for each $\widetilde{v}_{n},$
	$$\|\widetilde{v}_{n}\|_{\ell^{\infty}(\Omega)} \geq |v_{n}(y_{n})| \geq c >0.$$
	Since $V(x)$ and $f(x,u)$ are both T-periodic in $x$, $\Phi, \Phi'$ and $\mathcal{N}$ are invariant under the translation. Moreover, we have
	$$\Psi(\widetilde{v}_{n}) \to c,\quad \Psi'(\widetilde{v}_{n}) \to 0,$$
	$$\Phi(\widetilde{u}_{n}) \to c,\quad \Phi'(\widetilde{u}_{n}) \to 0.$$
	We may assume that $\|v_{n}\|_{\ell^{\infty}(\Omega)} \geq c >0.$
	Since $\Omega$ is bounded, there exists at least one point, say $x_{0}\in \Omega $, such that $v_{n}(x_{0}) \to v(x_{0}) \geq c >0$. Therefore, $|u_{n}(x_{0})|=\|u_{n}\||v_{n}(x_{0})| \to \infty$ as $n \to \infty.$ We have
	\begin{displaymath}
		\begin{aligned}
			0<\dfrac{\Phi(u_{n})}{\|u_{n}\|^{2}} 
			&=\dfrac{1}{2} - \int_{\mathbb{V}} \dfrac{F(x,u_{n})}{\|u_{n}\|^{2}} d\mu \\
			&\leq \dfrac{1}{2} - \dfrac{F(x_{0},u_{n})}{u_{n}^{2}(x_{0})} v_{n}^{2}(x_{0}) \, \to -\infty\quad \text{as } n \to \infty.
		\end{aligned}
	\end{displaymath}
	It is a contradiction, so that $\{u_{n}\}$ is bounded.
For bounded $\{u_{n}\}\subset H^{1}(\mathbb{Z}^{N})$, one has a subsequence $u_{n} \rightharpoonup u$ in $H^{1}(\mathbb{Z}^{N})$. By above discussions, it is pointwise convergent. For any $w \in C_{c}(\mathbb{Z}^{N})$,
\begin{displaymath}
	\begin{aligned}
		\Phi'(u_{n})w &= \,\langle u_{n},w\rangle - \int_{\mathbb{Z}^{N}} f(x,u_{n})w \, d\mu \\
		&\to \,\langle u,w\rangle - \int_{\mathbb{Z}^{N}} f(x,u)w \, d\mu \\
		&=	\Phi'(u)w. 		
	\end{aligned}
\end{displaymath}
We have $\Phi'(u)=0$, for $\Phi'(u_{n}) \to 0$. 
If $u\neq0$, then $u$ is a critical point of $\Phi$; otherwise, one can get a nontrivial critical point of $\Phi$ by translation.
In the case of $u=0$, suppose $u_{n} \to 0$ in $\ell^{q}(\mathbb{Z}^{N})$. For any $\epsilon >0$, there exists $C_{\epsilon}>0$ such that 
\begin{displaymath}
	\begin{aligned}
		\big| \int_{\mathbb{Z}^{N}} f(x,u_{n}) u_{n} d\mu \big| 
		&\leq \int_{\mathbb{Z}^{N}} |f(x,u_{n})||u_{n}| d\mu \\
		&\leq \int_{\mathbb{Z}^{N}} \big(\epsilon|u_{n}|+C_{\epsilon}|u_{n}|^{q-1}\big)|u_{n}| d\mu \\
		&=\epsilon\|u_{n}\|_{2}^{2}+C_{\epsilon}\|u_{n}\|^{q}_{q}
	\end{aligned}
\end{displaymath}
Since $ \|u_{n}\|_{2} \simeq \|u_{n}\|$, one can choose $\epsilon$ sufficiently small such that
\begin{displaymath}
	\begin{aligned}
		\Phi'(u_{n})u_{n}&=\|u_{n}\|^{2}-\int_{\mathbb{Z}^{N}} f(x,u_{n})u_{n} d\mu \\
		&\geq \frac{1}{2} \|u_{n}\|^{2}-C_{\epsilon}\|u_{n}\|_{q}^{q}.
	\end{aligned}
\end{displaymath}
Letting $n\to \infty$, then $\|u_{n}\|_{q} \to 0$ and  $\Phi'(u_{n})\to 0$ imply $\|u_{n}\| \to 0.$  This contradicts the fact that $\{u_{n}\} \subset \mathcal{N}$ is bounded away from $0$. Hence, 
$u_{n}  \nrightarrow 0$ in $\ell^{q}(\mathbb{Z}^{N}).$
By translation again, one can get a sequence $\{\widehat{u}_{n}\}$ such that
$$\|\widehat{u}_{n}\|_{\ell^{\infty}(\Omega)} \geq c >0,\quad,\forall n\in \mathbb{N},$$
for some constant $c>0$.
Since $\Phi, \Phi'$ and $\mathcal{N}$ are invariant under translation, then $\{\widehat{u}_{n}\} \subset \mathcal{N}$ and 
$$\Phi(\widehat{u}_{n})\to c,\quad 
\Phi'(\widehat{u}_{n})\to 0. $$ 
By passing to a subsequence, we have
\begin{equation}\notag
	\begin{aligned}
		&\widehat{u}_{n} \rightharpoonup \widehat{u} \quad \text{in }  H^{1}(\mathbb{Z}^{N}), \\
		&\widehat{u}_{n}(x)\to \widehat{u}(x),\quad\forall x\in \mathbb{Z}^{N},
	\end{aligned}
\end{equation}
and $\widehat{u} \neq 0.$ By the same argument, one has
$\Phi'(\widehat{u})=0,$ i.e., $\widehat{u}$ is a nontrivial critical point of $\Phi$ and $\widehat{u}\in\mathcal{N}$. 

Finally, we prove that the nontrivial critical point $u$ (or $\widehat{u}$) satisfies$$\Phi(u)=\inf_{\mathcal{N}}\Phi \equiv c.$$ We have an elementary estimate that, for any $u\neq 0$, $$F(x,u)\leq\frac{1}{2}f(x,u)u.$$ Combined with Fatou's Lemma, we have
\begin{displaymath}
	\begin{aligned}
		c+o(1) &= \Phi(u_{n}) -\frac{1}{2}\Phi'(u_{n})u_{n} \\
		&=\int_{\mathbb{Z}^{N}} \big( \frac{1}{2}f(x,u_{n})u_{n}-F(x,u_{n}) \big) d\mu \\
		&\geq \int_{\mathbb{Z}^{N}} \big( \frac{1}{2}f(x,u)u-F(x,u) \big) d\mu + o(1) \\
		&=\Phi(u)-\frac{1}{2}\Phi'(u)u +o(1)\\
		&=\Phi(u) + o(1),\quad \text{ as } n\to \infty.
	\end{aligned}
\end{displaymath}
Hence, $\Phi(u) \leq c \equiv \inf_{\mathcal{N}} \Phi.$ 
So that $\Phi(u)=c\equiv \inf_{\mathcal{N}}\Phi.$
\end{proof}
\begin{remark}
	Once we know $v_{n}  \nrightarrow 0$ in $\ell^{q}(\mathbb{R}^{N})$, one can use Concentration-Compactness Principle in the Euclidean space by P.L.Lions \cite{P.L.Lions, lions2, Struwe} to get a sequence $\{y_{n}\} \subset \mathbb{R}^{N}$ such that
	$$\int_{B_{r}(y_{n})} v_{n}^{2}\geq \delta$$
	for some $\delta>0$. 
	Inspired by this, we prove a similar result on lattice graphs to overcome the loss of compactness. In fact, we can find a sequence $\{y_{n}\}\subset \mathbb{Z^{N}}$ such that, for each $n\in \mathbb{N}$, $$|u_{n}(y_{n})| \geq c$$ for some constant $c>0$. This is sufficient for our proof.
\end{remark}

\section{Proof of  Theorem \ref{thm2} and Theorem \ref{thm3}}
    In this section, we first prove Theorem \ref{thm2} and then give a short proof for Theorem \ref{thm3}.
    \begin{proof}[Proof of Theorem \ref{thm2}]
	Consider the limit equation $ -\Delta u + V_{\infty}u = f(u)$
	and define its associated energy functional as
	$$\Phi_{\infty}(u):=\frac{1}{2}\int_{\mathbb{Z}^{N}} \big( |\nabla{u}|^{2} + V_{\infty}u^{2} \big) d \mu - \int_{\mathbb{Z}^{N}} F(u) d \mu .$$
	Note that the ground state energy of $\Phi$ can be characterized as 
	$$c=\inf_{\mathcal{N}} \Phi(u) 
	=\inf_{w\in H^{1} \backslash \{0\}}\max_{s>0} \Phi(sw). $$
	We denote the corresponding value for $\Phi_{\infty}$ by 
	$$c_{\infty}=\inf_{\mathcal{N_{\infty}}} \Phi_{\infty}(u) 
	= \inf_{w\in H^{1} \backslash \{0\}}\max_{s>0} \Phi_{\infty}(sw).$$
	Then $c_{\infty} \geq c$, for $V_{\infty} = \sup_{\mathbb{Z}^{N}} V(x).$ If $V(x)\equiv V_{\infty}$, then this is a special case of periodic potential. Otherwise, $V(x)<V_{\infty}$ on some subset of $\mathbb{V}$. By Theorem \ref{thm1}, we know that $c_{\infty}$ can be attained at some point $u\in \mathcal{N_{\infty}}$, i.e., $\Phi_{\infty}(u)=c_{\infty}.$  Then, for any $s>0$, 
	$$c_{\infty}=\Phi_{\infty}(u)\geq\Phi_{\infty}(su)>\Phi(su),$$ i.e.,
	$$c_{\infty}>\max_{s>0}\Phi(su)\geq \inf_{u\in H^{1} \backslash \{0\}}\max_{s>0}\Phi(su)\equiv c.$$
	As before, one can choose a minimizing sequence $\{w_{n}\}$ for $\Psi$ and assume that $\Psi'(w_{n}) \to 0$. Then we have $\Phi(u_{n}) \to c$ and $\Phi'(u_{n})\to 0$ with $u_{n}:= m(w_{n})$. Suppose $\|u_{n}\| \to \infty$. Letting $v_{n} = \dfrac{u_{n}}{\|u_{n}\|}$, by passing to a subsequence, 
	we have $v_{n} \rightharpoonup v$ in $H^{1}(\mathbb{Z}^{N})$ and $v_{n}(x) \to v(x)$ for all $x\in\mathbb{Z^{N}}$. A similar argument yields that $v_{n} \nrightarrow 0$ in $\ell^{q}(\mathbb{Z}^{N}).$ Again, by interpolation inequality, one has a subsequence $\{v_{n}\}$ and corresponding $\{y_{n}\}\subset \mathbb{Z}^{N}$ such that $|v_{n}(y_{n})| \geq \delta >0$ for all $n.$ If $\{y_{n}\}$ is bounded, we get a contradiction similar to the proof of Theorem \ref{thm1}. Otherwise, there exists a subsequence $|y_{n}| \to \infty$. Let $\widetilde{v}_{n}(x) := v_{n}(x-y_{n})$.  Then $\{\widetilde{v}_{n}\}$ is bounded, for $\|v_{n}\|=1.$ By passing to a subsequence, one can assume that $\widetilde{v}_{n} \rightharpoonup \widetilde{v}$ in $H^{1}(\mathbb{Z}^{N})$ and $\widetilde{v}_{n}(x) \to \widetilde{v}(x)$ for all $x\in\mathbb{Z^{N}}$ with $\widetilde{v} \neq 0.$ Hence,
	\begin{displaymath}
		\begin{aligned}
			0 \leq \dfrac{\Phi(u_{n})}{\|u_{n}\|^{2}}
			&=\frac{1}{2} - \int_{\mathbb{Z}^{N}} \dfrac{F(u_{n})}{u_{n}^{2}} v_{n}^{2} \,d\mu \\
			&=\frac{1}{2} - \int_{\mathbb{Z}^{N}} \dfrac{F(\widetilde{u}_{n})}{\widetilde{u}_{n}^{2}} \widetilde{v}_{n}^{2} \,d\mu \to -\infty,\quad \text{ as } n\to\infty.\\
		\end{aligned}
	\end{displaymath} 
	This is a contradiction, so that $\{\|u_{n}\|\}$ is bounded.
	
	By similar arguments, we obtain a subsequence $\{u_{n}\}$ and a corresponding sequence $\{z_{n}\} \subset \mathbb{Z}^{N}$ such that $|u_{n}(z_{n})| \geq \eta >0$ for all $n$. Therefore, $\widehat{u}_{n}\rightharpoonup \widehat{u} \neq 0$ with $\widehat{u}_{n}(x) := u_{n}(x-z_{n})$. It suffices to prove that $\{z_{n}\}$ is bounded.
	Suppose $|z_{n}| \to \infty$ as $n\to \infty$, we claim that $\Phi'_{\infty}(\widehat{u}) =0$. 
	In fact, for any $w\in C_{c}(\mathbb{Z}^{N})$, letting $w_{n}(x):=w(x-z_{n})$. By Ekeland's Variational Principle, we have
	\begin{displaymath}
		\begin{aligned}
			|\Phi'(u_{n})w_{n}|\leq \|\Phi'(u_{n})\|_{(H^{1})^{*}}\|w_{n}\|
			= \|\Phi'(u_{n})\| _{(H^{1})^{*}}\|w\| \to 0,\quad \text{ as } n\to \infty.
		\end{aligned}
	\end{displaymath} 
	 Moreover,
	\begin{displaymath}
		\begin{aligned}
			\Phi'(u_{n})w_{n}
			&= \int_{\mathbb{Z}^{N}} \Gamma(u_{n},w_{n})(x) + V(x)u_{n}(x)w_{n}(x)\, d\mu -\int_{\mathbb{Z}^{N}} f(u_{n})w_{n}(x)\, d\mu  \\
			&=\int_{\mathbb{Z}^{N}} \Gamma(\widehat{u}_{n}, w)(x) + V(x-z_{n})\widehat{u}_{n}(x)w(x)\, d\mu -\int_{\mathbb{Z}^{N}} f(\widehat{u}_{n})w(x)\, d\mu  \\
			&\to \int_{\mathbb{Z}^{N}} \Gamma(\widehat{u},w)(x) + V_{\infty}(x)\widehat{u}(x)w(x)\, d\mu -\int_{\mathbb{Z}^{N}} f(\widehat{u})w(x)\, d\mu  \\
			&=\Phi'_{\infty}(\widehat{u})w.
		\end{aligned}
	\end{displaymath} 
	Hence,
	\begin{equation}\notag
		\begin{aligned}
			c+o(1) &= \Phi(u_{n}) - \frac{1}{2}\Phi'(u_{n})u_{n}\\
			&=\int_{\mathbb{Z}^{N}} \Big( \frac{1}{2}f(u_{n})u_{n} - F(u_{n}) \Big) \,d\mu \\
			&=\int_{\mathbb{Z}^{N}} \Big( \frac{1}{2}f(\widehat{u}_{n})\widehat{u}_{n} - F(\widehat{u}_{n}) \Big) \,d\mu \\
			&\geq \int_{\mathbb{Z}^{N}} \Big( \frac{1}{2}f(\widehat{u})\widehat{u}- F(\widehat{u}) \Big) \,d\mu  +o(1)\\
			&=\Phi_{\infty}(\widehat{u})-\frac{1}{2}\Phi'_{\infty}(\widehat{u})\widehat{u}+o(1)\\
			&=\Phi_{\infty}(\widehat{u})+o(1)\\
			&\geq c_{\infty}+o(1),\quad n\to \infty.
		\end{aligned}
	\end{equation}
	This is a contradiction, so that $\{y_{n}\}$ is bounded.
 \end{proof}
Next, we prove the results on quasi-transitive graphs.
 \begin{proof}[Proof of Theorem \ref{thm3}]
 	Let $G/ \Gamma =\{\rho_{1},\cdots,\rho_{m}\}$ be the set of finitely many orbits,	$\Omega=\{z_{1},\cdots,z_{m}\}\subset\mathbb{V}$ where $z_{i}\in\rho_{i}, 1\leq i \leq m.$ Replace the translations in the proof of Theorem \ref{thm1} and \ref{thm2} by $\Gamma$-action on functions respectively. For example, suppose $|v_{n}(y_{n})|\geq \delta>0$. For each $y_{n}$, there exists $g_{n}\in\Gamma$ such that $g_{n}y_{n}=z_{i}\in \Omega$. Set $\widetilde{v}_{n}:=v_{n}\circ g_{n}^{-1}$, then we get a subsequence $\{\widetilde{v}_{n}\}$ such that $\|\widetilde{v}_{n}\|_{\ell^{\infty}(\Omega)} \geq \delta$ with $\Omega$ being a finite subset. By similar arguments, one can complete the proof on quasi-transitive graph $G$.
 \end{proof}
 
\section{Appendix}
In this section, we prove the propositions in Section 3.

\begin{proof}[Proof of Lemma \ref{lem1}]
	For any $u,v \in H^{1}(\mathbb{V}),$
	\begin{displaymath}
		\begin{aligned}
			\Phi ^{'}(u)v 
			&= \frac{1}{2} \frac{d}{d t} \Big|_{t=0} \big( \|u+tv\|^{2} \big) - \frac{d}{d t}\Big|_{t=0} I(u+tv)\\
			&=\,\langle u , v \rangle -  \frac{d}{d t}\Big|_{t=0} \Big( \sum_{x\in\mathbb{V}} F(x,u+tv)  \Big).  \\
		\end{aligned}
	\end{displaymath}
	In fact, since $u,v\in H^{1}(\mathbb{V}), \, H^{1}(\mathbb{V})\hookrightarrow  \ell^{q}(\mathbb{V}),  \, (q>2) $, one can exchange the order of derivative and summation. Then, we have
	$$
	\frac{d}{d t}\Big|_{t=0} \Big( \sum_{x\in\mathbb{V}} F(x,u+tv)   \Big) 
	= \sum_{x\in\mathbb{V}} \frac{d}{d t}\Big|_{t=0} F(x,u+tv) 
	=\sum_{x\in\mathbb{V}} f(x,u)v .  
	$$
	Hence,
	\begin{displaymath}
		\begin{aligned}
			\Phi ^{'}(u)v 
			&=\, \langle u , v \rangle - \sum_{x\in\mathbb{V}} f(x,u)v   \\
			&= \int_{\mathbb{V}} \Big( \Gamma(u,v)+ Vuv  \Big) d \mu - \int_{\mathbb{V}} f(x,u)v \,d \mu. \\
		\end{aligned}
	\end{displaymath}
	We conclude that $\Phi^{'}(u) = 0$ if and only if $u$ is a weak solution of (\ref{maineq}).
\end{proof}

\begin{proof}[Proof of Proposition \ref{prop1}]
	(i) For any given $\epsilon>0, f(x,u) = o(u)$ as $u \to 0$ and the growth condition imply that there exists $C_{\epsilon} > 0$ such that
	$$|f(x,u)| \leq \epsilon |u| + C_{\epsilon}|u|^{q-1}. $$
	Hence,
	\begin{displaymath}
		\begin{aligned}
			\dfrac{I^{'}(u)w}{\|u\|}
			&\leq \dfrac{1}{\|u\|} \int_{\mathbb{V}} \big( \epsilon |u| |w| + C_{\epsilon} |u|^{q-1}|w| \big) d \mu \\
			&\leq \dfrac{\epsilon}{\|u\|} \|u\|_{2}\|w\|_{2} 
			+ \dfrac{C_{\epsilon}}{\|u\|} \|u\|_{q}^{q-1}\|w\|_{q}                \\
			&\leq \epsilon \|w\|_{2} + C_{\epsilon} \|u\|^{q-2}\|w\|_{q},  \quad (q>2),		
		\end{aligned}
	\end{displaymath}
	where the second inequality follows from H\"{o}lder's inequality.
	Hence, 
	$$ \lim_{\|u\| \to 0} \dfrac{I^{'}(u)w}{\|u\|} \leq \epsilon \|w\|_{2},
	\quad  \forall w\in H^{1}(\mathbb{V}),$$
	Letting $\epsilon \to 0^{+}$, we have
	$$ \lim_{\|u\| \to 0} \dfrac{I^{'}(u)w}{\|u\|} = 0,\quad  \forall w\in H^{1}(\mathbb{V}).$$
	Hence,  $I'(u) = o(\|u\|)$. \\  
	(ii) By assumption, we know that $u\mapsto \frac{f(x,u)}{u}$ is strictly decreasing on $(-\infty, 0)$ and increasing on $(0, \infty)$. Since
	\begin{displaymath}
		\dfrac{I'(su)u}{s} = \int_{\mathbb{V}} \dfrac{f(x,su)u}{s} = \int_{\mathbb{V}} \dfrac{f(x,su)u^{2}}{su} ,
	\end{displaymath}
	and one can check that $s \mapsto \frac{f(x,su)}{su}$ is increasing on $(0,\infty)$ for any  $u\neq0$, the assertion is proved.	\\
	(iii) For any weakly compact subset $\mathcal{W} \subset H^{1}(\mathbb{V}) \backslash \{0\},$ it suffices to show that, for any $M>0,$ there exists $K>0$ such that 
	\begin{displaymath}
		\int_{\mathbb{V}} \dfrac{F(x,su)}{s^{2}} d \mu > M, \quad \forall u\in \mathcal{W},
	\end{displaymath}
holds for any $s>K$.
	Suppose it is not ture, there exist $M>0$, a sequence $s_{k} \to \infty$ and $\{u_{k}\} \subset \mathcal{W}$ such that
	\begin{displaymath}
		\int_{\mathbb{V}} \dfrac{F(x,s_{k}u_{k})}{s_{k}^{2}} d \mu \leq M.
	\end{displaymath}
	Since $\mathcal{W}$ is weakly compact, passing to a subsequence if necessary, $\{u_{n}\}$ is weakly convergent in $\mathcal{W} \subset H^{1}(\mathbb{V}) \backslash \{0\}.$ By choosing special test functions, we have $u_{k}(x) \to u(x)$ for all $x\in\mathbb{V}$  with $u\in \mathcal{W}\subset H^{1}(\mathbb{V}) \backslash \{0\}.$ That is, $u_{k}(x_{0}) \to u(x_{0}) \neq 0$ for some $x_{0} \in \mathbb{V}.$ One can see that $F$ is nonnegative. Hence,
	\begin{displaymath}
		\begin{aligned}
			M &\geq \int_{\mathbb{V}} \dfrac{F(x,s_{k}u_{k})}{s_{k}^{2}} d \mu \\
			&\geq \dfrac{F(x_{0},s_{k}u_{k}(x_{0}))}{s_{k}^{2}u_{k}^{2}(x_{0})} u_{k}^{2}(x_{0}) \, \to \infty \quad \text{ as }\, k\to \infty.
		\end{aligned}
	\end{displaymath}
This is a contradiction.
\end{proof}

\begin{proof}[Proof of Proposition \ref{prop2}]
	(i) Since
	$$\Phi(sw)= \frac{1}{2}\|w\|^{2}s^{2} - I(sw)$$
	and
	$$ \lim_{s\to 0^{+}} \dfrac{I(sw)}{\frac{1}{2}\|w\|^{2}s^{2} }
	=\lim_{s\to 0^{+}} \dfrac{I'(sw)w}{\|w\|^{2}s}
	=\lim_{s\to 0^{+}} \dfrac{I'(sw)}{\|sw\|}\dfrac{w}{\|w\|}=0,
	\quad \forall w\neq 0,
	$$
	we have  $\alpha_{w} (s) = \Phi(sw) >0$ when $s>0$ is suffciently small. \\
	Moreover,
	$$ \lim_{s\to \infty} \dfrac{I(sw)}{\frac{1}{2}\|w\|^{2}s^{2} }
	=\infty$$
	implies $$\alpha_{w} (s) = \Phi(sw) \to -\infty \, \,\,\text{as} \, s\to \infty.$$
	Consider its derivative 
	$$\alpha'_{w}(s)= \frac{d\,}{ds}\Phi(sw)=\|w\|^2s-I'(sw)w
	=s\Big(\|w\|^2-\frac{I'(sw)w}{s}\Big).$$
	Since $s \mapsto \dfrac{I'(sw)}{s}w$ is strictly increasing on $(0,\infty),$ i.e., $\alpha'_{w}(s)$ is strictly decreasing on $(0,\infty)$, there exists a unique $s_{w} >0$ such that $\alpha'_{w}(s)>0$ on $(0,s_{w})$ and $\alpha'_{w}(s)<0$ on  $(s_{w},\infty).$ \\
	(ii) We know that $s_{w}$ is the unique root for
	$$ \alpha'_{w}(s) = s \big( 1- \frac{I'(sw)}{s}w\big)=0.$$
	By $I'(u)= o(\|u\|) $ as $u \to 0$, we can choose small $\delta>0 $ such that
	$$\dfrac{I'(sw)w(x)}{s} \leq \dfrac{I'(sw)}{\|sw\|} < \frac{1}{2}, \quad  \,\, \forall s<\delta. $$
	i.e., $\alpha'_{w}(s)>\frac{1}{2} s >0$ for $s\in (0,\delta).$ So that $s_{w}\geq \delta$ for all $w\in S.$\\
	(iii) Moreover, from (i), one can see that $s_{w}$ is the unique maximum of $\alpha_{w}(s)$ attained before its zero point. For $w \in S, \,\alpha_{w}(s)=s^{2}\big( \frac{1}{2}-\frac{I(sw)}{s^{2}} \big)$ and $\frac{I(sw)}{s^{2}} \to \infty$ uniformly about $u \in \mathcal{K}$ as $s \to \infty.$
	We can choose $C_{\mathcal{K}}>0$ such that $\frac{I(sw)}{s^{2}} \geq 1$  for
	$s\geq C_{\mathcal{K}}$, $w \in \mathcal{K}$. Therefore, $\alpha_{w}(s)\leq -\frac{1}{2} C_{\mathcal{K}}^{2}<0$ whenever $s\geq C_{\mathcal{K}}$, i.e., $s_{w} \leq C_{\mathcal{K}} \text{ for all } w \in \mathcal{K}$.
\end{proof}

\begin{proof}[Proof of Proposition \ref{prop3}]
	(i) Suppose  that $w_{n} \to w \neq 0.$ Since $\widehat{m}(tw) = \widehat{m}(w)$ for each $t>0,$ we may assume that $\{w_{n}\} \subset S$. Set $\widehat{m}(w_{n}) = s_{n}w_{n}$. By Proposition \ref{prop2}, we know that $\{s_{n}\}$ is bounded away from 0. Hence, by taking a subsequence, we have $s_{n} \to \bar{s} >0$ and $\widehat{m}(w_{n}) \to \bar{s}w.$ Since $\mathcal{N}$ is closed, $\bar{s}w \in \mathcal{N}.$
	This implies that $\bar{s} = s_{w}$ and $\widehat{m}(w_{n}) \to \bar{s}w = s_{w}w = \widehat{m}(w)$.\\
	(ii) On one hand, $m : S \to \mathcal{N}$ is surjective, for $m(\frac{u}{\|u\|}) = u$ for each $u \in \mathcal{N}.$ On the other hand, if $w_{1}, w_{2}\in S$ and $m(w_{1})=m(w_{2}),$ then $s_{1}w_{1}=s_{2}w_{2}$. We have $s_{1}=s_{2}$, which implies $w_{1} = w_{2}$, i.e., $m$ is injective. Hence $m$ has an inverse mapping $m^{-1} : \mathcal{N} \to S.$ Denote the mapping $h: \mathcal{N} \to S ,\, u\mapsto  h(u)=\frac{u}{\|u\|}$. Then $h(m(w))=w=\text{id}(w)$ for all $w \in S$.
	Hence $m^{-1} = h, $ which is continuous.
\end{proof}

\begin{proof}[Proof of Proposition \ref{prop4}]
By definition,
\begin{displaymath}
		\begin{aligned}
			\widehat{\Psi}'(w)z
			&=\dfrac{d\,}{dt}\Big|_{t=0} \widehat{\Psi}(w+tz) \\
			&= \dfrac{d\,}{dt}\Big|_{t=0} \Phi(\widehat{m}(w+tz))\\
			&= \Phi'(\widehat{m}(w+tz))\Big|_{t=0} \cdot \dfrac{d\,}{dt}\Big|_{t=0} \widehat{m}(w+tz).
		\end{aligned}
	\end{displaymath}
	Since $\Phi \in C^{1}\big( H^{1}(\mathbb{V}),\mathbb{R} \big)$ and $\widehat{m}(w+tz) = s_{w+tz}(w+tz)$ is a continuous mapping, we have
	$$\Phi'\big(\widehat{m}(w+tz)\big)\Big|_{t=0}=\Phi'(\widehat{m}(w))$$
	and
	$$\dfrac{d\,}{dt}\Big|_{t=0} \widehat{m}(w+tz) = \dfrac{d\,}{dt}\Big|_{t=0} s_{w+tz}(w+tz)= s_{w}z.$$
Hence,	$$\widehat \Psi'(w)z = \Phi'(\widehat{m}(z))s_{w}z = s_{w}\Phi'(\widehat{m}(w))z = \frac{\| \widehat{m}(w)  \|}{\|w\|}\widehat{\Phi} '(\widehat{m}(w) )z.$$
\end{proof}

\begin{proof}[Proof of Proposition \ref{prop5}]
	Define a functional $\varphi : H^{1}(V) \to \mathbb{R}$,\, $\varphi(u) = \frac{1}{2} \|u\|^{2}$. Then 
	\begin{displaymath}
		\begin{aligned}
			D\varphi \equiv \varphi' : H^{1}(\mathbb{V})&\to(H^{1} )^{*}\\
			u  &\mapsto D\varphi(u) ,
		\end{aligned}
	\end{displaymath}
where
	\begin{displaymath}
	\begin{aligned}
		D\varphi(u) : H^{1} &\to \mathbb{R}\\
		v &\mapsto D_{v}\varphi(u) = \,<u,v>.
	\end{aligned}
\end{displaymath}
	Hence, $\varphi \in  C^{1}(H^{1}(\mathbb{V}),\mathbb{R})$ and $\varphi'(w)w=\|w\|^{2}=1$ for each $w\in S.$ Therefore, for each $w\in S,$ we have direct sum decomposition:
	$$H^{1}(\mathbb{V})=T_{w}S\oplus \mathbb{R}_{w}.$$
	Define the projection mapping as
	\begin{displaymath}
		\begin{aligned}
			P_{w} : H^{1}(\mathbb{V}) &\to T_{w}S \\
			z+tw&\mapsto z.
		\end{aligned}
	\end{displaymath}
	First of all, we check that $\{\|P_{w}\|\}$ has a uniform bound, where 
	$$ \|P_{w}\| = \sup_{z+tw\neq 0} \dfrac{\|P_{w}(z+tw)\|}{\|z+tw\|}.$$
	Note that
	$$\quad \|P_{w}(z+tw)\|=\|z\| \leq \|z+tw\| + |t|,\,(\|w\|=1).$$
	Moreover, for any $z\in T_{w}S := \{u:\varphi'(w)(u)=0\}$ and $\|w\| = 1$, we have $\varphi'(w)(z+tw)=\varphi'(w)z+t\varphi'(w)w=t$.  Therefore, 
	$$ |t| \leq \| \varphi'(w) \|\|z+tw\|.$$
	It suffices to show that $\{ \| \varphi'(w)\| \}$ has a uniform bound for all $w\in S.$
	Recall that
	$$\|\varphi'(w)\| = \sup \Big\{ |\varphi'(w)v|:\|v\|=1 \Big\}.$$
	and
	\begin{equation}\label{secd}
		\begin{aligned}
			\big|\varphi'(w)v\big|
			&=\big| \sum_{x\in\mathbb{V}}w(x)v(x)+\frac{1}{2}\sum_{x\in\mathbb{V}}\sum_{y\sim x}\big( w(y)-w(x) \big) \big( v(y)-v(x) \big) \big| \\
			 &\leq \|w\|_{2}\|v\|_{2}+\frac{1}{2}\sum_{x\in\mathbb{V}}\sum_{y\sim x}\big|w(y)-w(x)\big|\big( |v(y)|+|v(x)|\big).
		\end{aligned}
	\end{equation}
	For the second term in the last line of (\ref{secd}),
	\begin{displaymath}
	\begin{aligned}
		&\frac{1}{2}\sum_{x\in\mathbb{V}}\sum_{y\sim x}\big|w(y)-w(x)\big|\big( |v(y)|+|v(x)|\big)\\
		&\leq \sum_{x\in\mathbb{V}}\sum_{y\sim x} |w(y)||v(x)|
		+\sum_{x\in\mathbb{V}}\sum_{y\sim x} |w(x)||v(x)| \\
		&\leq \big(\sum_{x\in\mathbb{V}}\sum_{y\sim x} |w(y)|^{2}\big)^{1/2}\big(\sum_{x\in\mathbb{V}}\sum_{y\sim x} |v(x)|^{2}\big)^{1/2}
		+C\|w\|_{2}\|v\|_{2}\\
		&\leq 2C\|w\|_{2}\|v\|_{2}.
	\end{aligned}
\end{displaymath}
	Therefore,
	$$\|z\| \leq \|z+tw\|+|t|\|w\| \leq (1+C')\|z+tw\|,$$
	where $C'$ is a positive constant.
	We show that
	\begin{equation}\label{*}
		\| \Psi'(w)\| \leq \|u\|\|\Phi'(u)\| \leq (C'+1)\|\Psi'(w)\|,
	\end{equation}
	where $w\in S,\,u=m(w)\in \mathcal{N} \text{ and }\|u\| \geq \delta >0.$\\
	On one hand, 
	\begin{displaymath}
		\begin{aligned}
			\| \Psi'(w)\|  &= \sup_{\mbox{\tiny$\begin{array}{c}
						z\in T_{w}S \\
						\|z\|=1
					\end{array}$}} | \Psi'(w)z|    \\
			&= \sup_{\mbox{\tiny$\begin{array}{c}
						z\in T_{w}S \\
						\|z\|=1
					\end{array}$}} \big| \Phi'(m(w))z \big|\cdot \|m(w)\|    \\
			&=  \sup_{\mbox{\tiny$\begin{array}{c}
						z\in T_{w}S \\
						\|z\|=1
					\end{array}$}} \big| \Phi'(u)z \big|\cdot \|u\|    \\
			&\leq \|u\|\cdot\|\Phi'(u)\|
		\end{aligned}               
	\end{displaymath}
	On the other hand, 
	\begin{displaymath}
		\begin{aligned}
			\|u\|\| \Phi'(u)\|  &=\|u\| \sup_{\mbox{\tiny$\begin{array}{c}
						z+tw \neq 0 \\
						z\in T_{w}S 
					\end{array}$}} \dfrac{\big| \Phi'(u)(z+tw)\big| }{\|z+tw\|}   \\
			&\leq \|u\|(1+C')\sup_{z\in T_{w}S} \dfrac{\big|\Phi'(u)z\big|}{\|z\|}.   
		\end{aligned}               
	\end{displaymath}
	Since $\Phi'(u)w = \frac{1}{\|u\|} \Phi'(u)u = 0$ and $u=m(w),$ we have
	\begin{displaymath}
		\begin{aligned}
			\|u\|\| \Phi'(u)\|  &=(1+C')\sup_{z\in T_{w}S} \dfrac{\|m(w)\| \big| \Phi'(m(w))z\big| }{\|z\|}  \\
			&=(1+C') \sup_{z\in T_{w}S}\dfrac{\big|\Psi'(w)z\big|}{\|z\|}   \\
			&=(1+C') \| \Psi'(w)\|.
		\end{aligned}               
	\end{displaymath}
	If $\{u_{n}\}$ is a bounded \textit{Palais-Smale sequence} for $\Phi$, that is, $\Phi'(u_{n})\to 0$ and $ \delta \leq \|u_{n}\| \leq C$  for some constant $0<\delta<C,$  then $\Psi(w_{n})\equiv \Psi\big( \frac{u_{n}}{\|u_{n}\|} \big)$ is bounded as well. Moreover,
	$$\|\Psi'(w_{n})\| \leq \|u_{n}\|\|\Phi'(u_{n})\| \leq C\|\Phi'(u_{n})\|  \to 0,\quad n\to \infty.$$ 
	We conclude that $\{w_{n}\}$ is a bounded \textit{Palais-Smale sequence} for $\Psi$.
	By (\ref{*}) one can see that $\|\Psi'(w_{0})\| = 0 $ if and only if $\|\Phi'(u_{0})\| = 0$ with $u_{0} = m(w_{0}).$ By definition, $\Phi(u_{0}) = \Psi(w_{0})$. Since $\Phi \big|_{\mathcal{N}} >0$ is bounded away from below, it has an infimum and $\inf_{\mathcal{N}} \Phi= \inf_S \Psi $.
\end{proof}

\textbf{Acknowledgements.}
The authors would like to thank Zhang Peng for helpful discussions and suggestions.

\bibliography{ref}
\bibliographystyle{alpha}

\end{document}